\documentclass[12pt,letterpaper]{amsart}
\usepackage{amsmath,amssymb}
\usepackage{amsthm}
\usepackage[abbrev]{amsrefs}
\usepackage{latexsym}
\usepackage{hyperref}
\usepackage{color}

\numberwithin{equation}{section}
\theoremstyle{plain}
\newtheorem{thm}{Theorem}[section]
\newtheorem{prop}[thm]{Proposition}
\newtheorem{lem}[thm]{Lemma}

\theoremstyle{definition} 
\newtheorem{dfn}[thm]{Definition}

\newtheorem*{dfn*}{Definition}

\theoremstyle{remark}
\newtheorem{rem}[thm]{Remark}
\newtheorem*{ack}{Acknowledgment}
\newcommand{\set}[1]{\{\,{#1}\,\}}

\DeclareMathOperator{\pr}{pr}
\DeclareMathOperator{\id}{id}

\newcommand{\cL}{\mathcal{L}}

\newcommand{\field}[1]{\mathbb{#1}}

\newcommand{\R}{\field{R}}

\newcommand{\N}{\field{N}}

\newcommand{\ep}{\varepsilon}

\newcommand{\mmsp}{mm-space}
\newcommand{\supp}{\mathop{\rm supp}}
\newcommand{\diam}{\mathop{\rm diam}}

\DeclareMathOperator{\dP}{{\it d}_{{\rm P}}}

\newcommand{\dkf}[1]{{\it d}_{\rm KF}^{#1}}

\DeclareMathOperator{\dconc}{{\it d}_{{\rm conc}}}
\newcommand{\dconcpi}[1]{{\it d}_{{\rm conc}}^{#1}}
\newcommand{\Lip}{\mathcal{L}{\it ip}}

\newcommand{\dx}{{d_X}}

\newcommand{\mux}{{m_X}}
\newcommand{\muy}{{m_Y}}
\newcommand{\leb}{{\mathcal L^1}}

\DeclareMathOperator{\dis}{dis}

\title{Box distance and observable distance via optimal transport}
\author{Hiroki Nakajima}
\address{Institute for excellence in higher education, Tohoku University, Sendai 980-8576, Japan}
\email{hiroki.nakajima.a1@tohoku.ac.jp}
\date{\today}
\thanks{The author was supported by JSPS KAKENHI Grant Number 19J10866}
\begin{document}
\maketitle
\begin{abstract}
On the set of all metric measure spaces, we have two important metrics, the box metric and the observable metric, both introduced by M. Gromov. We obtain the representation of these metrics by using transport plan. In addition, we prove the existence of optimal transport plans of these metrics.
\end{abstract}
\tableofcontents
\section{Introduction}

In the metric measure geometry, the box distance and the observable distance are two of the most important distances between two metric measure spaces. The box metric $\square$ is elementary and equivalent to the Gromov-Prohorov metric $d_{\rm GP}$. More precisely,  $d_{\rm GP}\le \square \le 2d_{\rm GP}$ holds \cite{Loh:dGP}.

The observable distance is defined based on the concentration of measure phenomenon. Its topology is weaker than the topology given by the box distance (see Proposition 5.5 in \cite{Shioya:mmg}). These distances are introduced by M. Gromov in \cite{Gmv:green}. 

In this paper, we obtain the following Theorem \ref{thm:boxOpt}. That is a simple representation of box distance $\square(X,Y)$ between two metric measure spaces $X$ and $Y$ by using optimal transport plans (optimal couplings). The definition of the box metric $\square$ is Definition \ref{dfn:box}. We assume that all metric spaces are complete and separable. We also assume that all metric measure spaces are equipped with Borel probability measures. The symbol $m_X$ denotes the measure of metric measure space $X$ and the symbol $d_X$ the metric of $X$. 
We define the {\it distortion} of a nonempty subset $S\subset X\times Y$ by
\[
\dis S:=\sup\set{|d_X(x,x')-d_Y(y,y') | \ ;\  (x,y),(x',y')\in S }
\]
and $\dis \emptyset:=\infty$.
\begin{thm}\label{thm:boxOpt}
\[
\square (X,Y)=\min_{\pi,S}\max\{1-\pi(S), \dis S \},
\]
where $\pi$ runs over all transport plans between $m_X$ and $m_Y$, and where $S$ runs over all closed subsets of $X\times Y$.
\end{thm} 
Theorem \ref{thm:boxOpt} claims the existence of an optimal transport plan. 
This follows from the weak compactness of the set of all transport plans between $m_X$ and $m_Y$, and the continuity of the map $\pi \mapsto \inf_S\max\{1-\pi(S),\dis S\}$.
The existence of a minimizer of the function $S\mapsto\max\{1-\pi(S),\dis S\}$ also follows because this function is lower semi continuous with respect to weak Hausdorff convergence. The proof of Theorem \ref{thm:boxOpt} is given in Section \ref{sec:boxOpt}.

The definition of the Eurandom metric $d_{\rm Eur}$ is very similar to the right hand side of Theorem \ref{thm:boxOpt}. The optimal transports of the Eurandom metric $d_{\rm Eur}$ is explained in Section \ref{sec:Eur_opt}. The Eurandom metric $d_{\rm Eur}$ gives the topology of the box metric $\square$ and $d_{\rm Eur}\le 2d_{\rm GP}$ holds \cite{GPW}. However, there is no constant $C>0$ that satisfies $d_{\rm GP}\le Cd_{\rm Eur}$, because the Gromov-Prohorov metric $d_{\rm GP}$ is complete but the Eurandom metric $d_{\rm Eur}$ is not.
The relation to other similar metric is summarized in Subsection 2.3 in \cite{Sturm:sp}. Note that the Eurandom metric $d_{\rm Eur}$ is called the $L^0$-distortion metric in \cite{Sturm:sp}.

An application of Theorem \ref{thm:boxOpt} is to simplify the proof of the following theorem.
\begin{thm}[Chapter 3$\frac12_+$ in \cite{Gmv:green}]\label{thm:boxIsMetric}
The function $\square$ is a metric on the set of all isomorphism classes of metric measure spaces.
\end{thm}
We obtained a new proof of the nondegeneracy of the box metric. The proof uses the existence of optimal transport, which makes the policy easy to understand.
The original proof in \cite{Gmv:green} uses the distance matrix distribution (cf. Theorem 4.10 in \cite{Shioya:mmg}).

For the observable distance, we also obtain a representation by using optimal transport.
See Theorem \ref{thm:optConc} for its representation. This also simplifies the proof of the triangle inequality and nondegeneracy for the observable metric. See Theorems \ref{thm:dconcNonDeg} and \ref{thm:obsTriangle} for details.

The box distance and the observable distance are extended to distances between two metric measure spaces such that the group acts on them \cite{NkjShioya:mm-gp}. The proof of the nondegeneracy of these distances is an application of the proof in the present paper.

\begin{ack}
The author would like to thank Professor Takashi Shioya for many helpful suggestions. He also thanks Dr. Daisuke Kazukawa and Dr. Shinichiro Kobayashi for many stimulating discussions.
\end{ack}

\section{Preliminaries}\label{preliminaries}

In this section, we present some basics of mm-space.
We refer to \cite{Gmv:green,Shioya:mmg} for more details about the contents of this section.
\begin{dfn}[\mmsp]
Let $(X,\dx)$ be a complete separable metric space and $m_X$ a Borel probability measure on $X$.
We call such a triple $(X,\dx,\mux)$ an {\it\mmsp}.
We sometimes say that $X$ is an \mmsp, for which the metric and measure of $X$ are respectively indicated by $\dx$ and $\mux$. 
\end{dfn}

We denote the Borel $\sigma$-algebra over $X$ by $\mathcal B_X$.
For any point $x\in X$, any two subsets $A,B\subset X$ and any real number $r\ge 0$, we define
\begin{align*}
d_X(x,A)&:=\inf\set{d_X(x,y)\mid y\in A} ,\\
U_r^{d_X}(A)&:=\set{y\in X\mid d_X(y,A)<r},
\end{align*}
where $\inf \emptyset:=\infty$. The symbol $U_r^{d_X}(A)$ is often omitted as $U_r(A)$. We remark that $U_r(\emptyset)=\emptyset$ for any real number $r\ge 0$.  The {\it diameter of $A$} is defined by $\diam A:=\sup_{x,y\in A}d_X(x,y)$ for $A\neq\emptyset$ and $\diam \emptyset:=0$.

Let $Y$ be a topological space and let $p:X\to Y$ be a measurable map from a measure space $(X,\mux)$ to a measurable space $(Y,\mathcal B_Y)$. {\it The push-forward $p_*m_X$ of $\mux$ by the map $p$} is defined as $p_*\mux(A):=\mux(p^{-1}(A))$ for any $A\in \mathcal B_Y$.

\begin{dfn}[support]
Let $\mu$ be a Borel measure on a topological space $X$.
We define the support $\supp\mu$ of $\mu$ by
\[
\supp\mu:=\{x\in X\mid \text{ $m_X(U)>0$ for any open neighborhood $U$ of $x$} \}.
\]
\end{dfn}
\begin{prop}
Let $X$ and $Y$ be two topological spaces and let $f\colon X\to Y$ be a continuous map. If a Borel measure $\mu$ on $X$ satisfies 
\[
\mu(X\setminus \supp\mu)=0,
\]
then we have
\[
\supp f_*\mu=\overline{f(\supp\mu)}.
\]
\end{prop}
\begin{proof}
Since
\begin{align*}
f_*\mu(Y\setminus \overline{f(\supp\mu)})&=\mu(X\setminus f^{-1}(\overline{f(\supp\mu)}))\\
&\le \mu(X\setminus \supp\mu)=0,
\end{align*}
we have $\supp f_*\mu\subset \overline{f(\supp\mu)}$.

Next, let us prove
\begin{equation}\label{prop:push:goal}
f(\supp \mu)\subset \supp f_*\mu.
\end{equation}
Take any $y\in f(\supp\mu)$ and any open neighborhood $U$ of $y$.
Since $y\in f(\supp\mu)$, there exists $x\in\supp\mu$ such that $y=f(x)$.
The set $f^{-1}(U)$ is an open neighborhood of $x\in\supp\mu$ because $f$ is continuous. Then we have
$
f_*\mu(U)= \mu(f^{-1}(U))>0.
$
This implies that $y\in\supp f_*\mu$ and we obtain \eqref{prop:push:goal}.
Since the set $\supp f_*\mu$ is closed, we have
\[
\overline{f(\supp\mu)}\subset \supp f_*\mu.
\]
This completes the proof.
\end{proof}
\begin{rem}
Let $\mu$ be a Borel measure $\mu$ on a topological space $X$. We have $\mu(X\setminus \supp \mu)=0$ if $X$ is second countable. Even if we assume that $\mu$ is inner regular, the equation $\mu(X\setminus \supp \mu)=0$ holds.
\end{rem}
\begin{dfn}[mm-isomorphism]
Two \mmsp s $X$ and $Y$ are said to be {\it mm-isomorphic} if there exists an isometry $f:\supp\mux\to\supp\muy$ such that $f_*\mux=\muy$, where $\supp\mux$ is the support of $\mux$.
Such an isometry $f$ is called an {\it mm-isomorphism}.
The mm-isomorphism relation is an equivalence relation on the class of mm-spaces.
Denote by $\mathcal X$ the set of mm-isomorphism classes of \mmsp s.
\end{dfn}
\begin{dfn}[Lipschitz order]\label{def:Lip_ord}
Let $X$ and $Y$ be two \mmsp s.
We say that $X$ {\it dominates} $Y$ and write $Y\prec X$ if there exists a 1-Lipschitz map $f:X\to Y$ satisfying
\[
f_*\mux=\muy.
\]
We call the relation $\prec$ on $\mathcal X$ the {\it Lipschitz order}.
\end{dfn}
\begin{prop}[Proposition 2.11 in \cite{Shioya:mmg}]\label{prop:lipPartialOrder}
The Lipschitz order $\prec$ is a partial order relation on $\mathcal X$.
\end{prop}

\begin{dfn}[Transport plan]\label{dfn:trans}
Let $\mu$ and $\nu$ be two Borel probability measures on $\R$. We say that a Borel probability measure $\pi$ on $\R^2$ is a transport plan between $\mu$ and $\nu$ if we have $(\pr_1)_*\pi=\mu$ and $(\pr_2)_*\pi=\nu$, where $\pr_1$ and $\pr_2$ are the first and second projections respectively. We denote by $\Pi(\mu,\nu)$ the set of transport plans between $\mu$ and $\nu$.
\end{dfn}


\subsection{Box distance and observable distance}
In this section, we briefly describe the box distance function and the observable distance function.
\begin{dfn}[Parameter]
Let $I:=[0,1)$ and let $\mathcal L^1$ be the Lebesgue measure on $I$.
Let $X$ be a topological space equipped with a Borel probability measure $\mux$.
A map $\varphi:I\to X$ is called a {\it parameter of $X$} if $\varphi$ is a Borel measurable map such that
\[
\varphi_*\mathcal L^1=\mux.
\]
\end{dfn}
\begin{dfn}[Pseudo-metric]
A {\it pseudo-metric $\rho$ on a set $S$} is defined to be a function $\rho:S\times S\to[0,\infty)$ satisfying
\begin{enumerate}
\item $\rho(x,x)=0,$
\item $\rho(y,x)=\rho(x,y),$
\item $\rho(x,z)\leq\rho(x,y)+\rho(y,z)$
\end{enumerate}
for any $x,y,z\in S$.
\end{dfn}

If $\rho$ is a metric, $\rho(x,y)=0$ implies $x=y$ for any two points $x,y\in S$. However, a pseudo-metric is not necessary to satisfy this condition. 
\begin{dfn}[Box distance]\label{dfn:box}
For two pseudo-metrics $\rho_1$ and $\rho_2$ on $I:=[0,1)$, we define $\square(\rho_1,\rho_2)$ to be the infimum of $\varepsilon\geq 0$ such that there exists a Borel subset $I_0\subset I$ satisfying
\begin{enumerate}
\item $|\rho_1(s,t)-\rho_2(s,t)|\leq\varepsilon$ for any $s,t\in I_0$,
\item $\mathcal L^1(I_0)\geq 1-\varepsilon$.
\end{enumerate}
We define the {\it box distance $\square(X,Y)$ between two \mmsp s $X$ and $Y$} to be the infimum of $\square(\varphi^*d_X,\psi^*d_Y)$, where $\varphi:I\to X$ and $\psi:I\to Y$ run over all parameters of $X$ and $Y$, respectively, and where $\varphi^*\dx(s,t):=\dx(\varphi(s),\varphi(t))$ for $s,t\in I$.
\end{dfn}
\begin{dfn}[Prohorov metric]
The Prohorov metric $\dP$ is defined by
\[
\dP(\mu,\nu):=\inf\set{\ep>0\mid \mu(U_\ep(A))\ge \nu(A)-\ep \text{ for any Borel set $A\subset X$}}
\]
for any two Borel probability measures $\mu$ and $\nu$ on a metric space $X$.
\end{dfn}
\begin{dfn}[Ky Fan metric]
Let $(X,\mu)$ be a measure space. For two $\mu$-measurable maps $f,g:X\to\R$, we define the {\it Ky Fan metric} $d_{\rm KF}=d_{\rm KF}^\mu$ by
\[
d_{\rm KF}^\mu(f,g):=\inf\set{\ep\ge 0\mid \mu(\set{x\in X\mid |f(x)-g(x)|>\ep})\le\ep}.
\]
\end{dfn}


\begin{dfn}[Observable distance]

For a parameter $\varphi$ of an mm-space $X$, we define 
\[
\Lip_1(X):=\set{f:X\to \R\mid \text{ $f$ is $1$-Lipschitz}}
\]
and 
\[
\varphi^*\Lip_1(X):=\set{f\circ \varphi\mid f\in \Lip_1(X)}.
\]
The Hausdorff distance function $d_{\rm H}^{\rm KF}$ is defined with respect to $d_{\rm KF}^{\cL^1}$.
We define the {\it observable distance} $\dconc$ between two mm-spaces $X$ and $Y$ by
\[
\dconc(X,Y):=\inf_{\varphi,\psi} d_{\rm H}^{\rm KF}(\varphi^*\Lip_1(X),\psi^*\Lip_1(Y)),
\]
where $\varphi:I:=[0,1)\to X$ and $\psi:I\to Y$ run over all parameters of $X$ and $Y$ respectively.
\end{dfn}

\section{Prohorov distance via optimal transport}
In this section, our goal is to prove the following Theorem \ref{thm:prohOpt}.
The proof of Theorem \ref{thm:prohOpt} is similar to the proof of Theorem \ref{thm:boxOpt}, so it is used as a reference in the proof of Theorem \ref{thm:boxOpt}. Note that Theorem \ref{thm:prohOpt} will not be used in subsequent sections.
\begin{thm}\label{thm:prohOpt}
Let $X$ be a metric space and $\mu, \nu$ be two Borel probability measures on $X$.
Then we have
\[
\dP(\mu,\nu)=\min_{\pi\in\Pi(\mu,\nu), S\subset X\times X} \max\{\dis_\Delta S,1-\pi(S)\},
\]
where $S\subset X\times X$ runs over all closed sets, and where we put
\[
\dis_\Delta(S):=\sup\{d_X(x,y)\mid (x,y)\in S\}
\]
for a nonempty subset $S\subset X\times X$, and $\dis_\Delta \emptyset:=0$.
\end{thm}

Theorem \ref{thm:prohOpt} claims the existence of a minimizer for $\pi$ and $S$ in the following Theorem \ref{thm:Strassen}.
\begin{thm}[Strassen's theorem]\label{thm:Strassen}
Let $X$ be a metric space and $\mu, \nu$ be two Borel probability measures on $X$.
Then we have
\[
\dP(\mu,\nu)=\inf_{\pi\in\Pi(\mu,\nu), S\subset X\times X} \max\{\dis_\Delta S,1-\pi(S)\},
\]
where $S\subset X\times X$ runs over all closed sets.
\end{thm}
First, we show the existence of an optimal transport plan $\pi$.
In preparation, we introduce the following definition.
\begin{dfn}[Distorsion for transport plan]\label{dfn:disDelta}
Let $\pi$ be a Borel probability measure on a metric space $X$. We define
\[
\dis_\Delta \pi:=\inf_{S\subset X\times X} \max\{\dis_{\Delta}S,1-\pi(S)\},
\]
where $S$ runs over all closed subsets over $X\times X$. 
\end{dfn}
In order to show the existence of an optimal transport plan $\pi$, it is sufficient to prove the lower semicontinuity of the map $\pi\mapsto \dis_\Delta \pi$. By Lemma \ref{lem:disDiagonalConti} below, we see that this map is continuous.
\begin{lem}\label{lem:disDiagonalIneq}
Let $X$ be a metric space. We assume that the product space $X\times X$ is equipped with the $l^1$-metric.
Then we have
\[
\dis_\Delta U_t(S)\le \dis_\Delta S+t.
\]
\end{lem}
\begin{proof}
Take any $(x,y)\in U_t(S)$. There exists $(x',y')\in S$ such that $d_X(x,x')+d_X(y,y')<t$.
We have
\begin{align*}
d_X(x,y)&\le d_X(x,x')+d_X(x',y')+d_X(y',y)\\
&<d_X(x',y')+t\\
&\le \dis_\Delta S+t.
\end{align*}
This completes the proof.
\end{proof}

\begin{lem}\label{lem:disDiagonalConti}
Let $\pi$ and $\pi'$ be two Borel probability measures on the product space $X\times X$ of an mm-space $X$, where $X\times X$ is equipped with the $l^1$-metric. Then we have
\[
|\dis_\Delta\pi-\dis_\Delta\pi'|\le \dP(\pi,\pi').
\]
\end{lem}

\begin{proof}
Take any real number $t>\dP(\pi,\pi')$ and any Borel set $S\subset X\times X$.
Since $t>\dP(\pi,\pi')$, we have
\[
\pi(U_t(S))\ge\pi'(S)-t.
\]
By Lemma \ref{lem:disDiagonalIneq}, we have
\begin{align*}
\dis_\Delta\pi&\le \max\{\dis_\Delta U_t(S),1-\pi(U_t(S))\}\\
&\le \max\{\dis_\Delta S,1-\pi'(S)\}+t.
\end{align*}
By the arbitrariness of $S\subset X\times X$, we obtain
\[
\dis_\Delta \pi\le\dis_\Delta \pi'+t.
\]
By exchanging $\pi$ for $\pi'$, we also obtain $\dis_\Delta \pi'\le\dis_\Delta \pi+\dP(\pi,\pi')$.
\end{proof}

\begin{lem}\label{lem:disDeltaPiOpt}
Let $\mu$ and $\nu$ be two Borel probability measures on a complete separable metric space $X$. Then there exists a transport plan $\pi\in\Pi(\mu,\nu)$ such that $\dP(\mu,\nu)=\dis_\Delta \pi$.
\end{lem}
\begin{proof}
This follows from Lemma \ref{lem:disDiagonalConti} and the weak compactness of $\Pi(\mu,\nu)$.
\end{proof}
We now prepare to show the existence of a minimizer of
\[
\inf_{S\subset X\times X}\max\{\dis_{\Delta}S,1-\pi(S)\}.
\]
Let us prepare some lemmas using the following weak Hausdorff convergence.
\begin{dfn}[Weak Hausdorff convergence]\label{dfn:wH}
Let $(X,d_X)$ be a metric space. Let $A,A_n\subset X$, $n=1,2,\dots$, be closed subset of $X$. We say that {\it $A_n$ converges weakly to $A$ as $n\to\infty$} if the following \eqref{dfn:wH1} and \eqref{dfn:wH2} are both satisfied.
\begin{enumerate}
\item For any $x\in A$, we have
\[
\lim_{n\to\infty}d_X(x,A_n)=0.
\]\label{dfn:wH1}
\item For any $x\in X\setminus A$, we have
\[
\liminf_{n\to\infty}d_X(x,A_n)>0.
\]\label{dfn:wH2}
\end{enumerate}
This convergence is called the {\it weak Hausdorff convergence} or the {\it Kuratowski-Painlev\'e convergence}. 
\end{dfn}

For the weak Hausdorff convergence, the following Theorem \ref{thm:weakHcpt} is useful.
\begin{thm}[Theorem 5.2.12 in \cite{Beer:Top}]\label{thm:weakHcpt}
Let $X$ be a second countable metric space. For any sequence of closed sets of $X$, there exists a convergent subsequence in the weak Hausdorff sense.
\end{thm}

\begin{lem}\label{lem:measLSC}
Let $\mu$ be a Borel probability measure on a metric space $X$. Let $A,A_n\subset X$, $n=1,2,\dots$, be closed subsets of $X$. If $A_n$ converges weakly to $A$ as $n\to\infty$, then we have
\[
\mu(A)\ge \limsup_{n\to\infty}\mu(A_n).
\]
\end{lem}
\begin{proof}
We have
\begin{align*}
\{x\in X \mid \liminf_{n\to\infty} d_X(x,A_n)=0\}&=\bigcap_{\varepsilon>0}\bigcap_n\bigcup_{k\ge n}U_\varepsilon(A_k)\\
&\supset  \bigcap_n\bigcup_{k\ge n} A_k.
\end{align*}
Since $A_n$ converges weakly to $A$ as $n\to\infty$, we have
\[
A\supset \{x\in X \mid \liminf_{n\to\infty} d_X(x,A_n)=0\}.
\]
Therefore we obtain
\begin{align*}
\mu(A)&\ge \mu(\bigcap_n\bigcup_{k\ge n} A_k)\\
&=\lim_{n\to\infty}\mu(\bigcup_{k\ge n} A_k)\\
&\ge \limsup_{n\to\infty}\mu(A_n).
\end{align*}
\end{proof}

\begin{lem}\label{lem:disDeltaLSC}
Let $X\times X$ be the $l^1$-product space of a metric space $X$. Let $S,S_n\subset X\times X$, $n=1,2,\dots$, be closed subsets. If $S_n$ converges weakly to $S$ as $n\to\infty$, then we have
\[
\dis_\Delta(S)\le \liminf_{n\to\infty}\dis_\Delta(S_n).
\]
\end{lem}
\begin{proof}
Since $S_n$ converges weakly to $S$ as $n\to\infty$, we have
\begin{align*}
S&\subset \left\{(x,x')\in X\times X\mid \lim_{n\to\infty}d_{X\times X}((x,x'),S_n)=0\right\}\\
&=\bigcap_{\varepsilon>0}\bigcup_n\bigcap_{k\ge n}U_\varepsilon(S_k).
\end{align*}
By Lemma \ref{lem:disDiagonalIneq}, we have
\begin{align*}
\dis_\Delta S&\le \dis_\Delta\left(\bigcap_{\varepsilon>0}\bigcup_n\bigcap_{k\ge n}U_\varepsilon(S_k)\right)\\
&\le \lim_{\varepsilon\to +0}\dis_\Delta\left(\bigcup_n\bigcap_{k\ge n}U_\varepsilon(S_k)\right)\\
&=\lim_{\varepsilon\to +0}\sup_n\dis_\Delta\left(\bigcap_{k\ge n}U_\varepsilon(S_k)\right)\\
&\le \lim_{\varepsilon\to +0}\sup_n\inf_{k\ge n}\dis_\Delta\left(U_\varepsilon(S_k)\right)\\
&\le \lim_{\varepsilon\to +0} \liminf_{n\to\infty}(\dis_\Delta S_n+\varepsilon)\\
&=\liminf_{n\to\infty}\dis_\Delta S_n.
\end{align*}
\end{proof}

\begin{lem}\label{lem:disDeltaS_Opt}
Let $\pi$ be a Borel probability measure on $X\times X$. Then there exists a closed set $S\subset X\times X$ such that
\[
\dis_\Delta \pi=\max\{\dis_\Delta S,1-\pi(S)\}.
\]
\end{lem}

\begin{proof}
By the definition of $\dis_\Delta \pi$, there exists a sequence $\{S_n\}$ of closed sets of $X\times X$ such that
\[
\max\{\dis_\Delta(S_n),1-\pi(S_n)\}<\dis_\Delta\pi+\frac 1n.
\]
By Theorem \ref{thm:weakHcpt}, there exist an increasing function $\iota:\N\to\N$ and a closed set $S\subset X\times X$ such that $S_{\iota(n)}$ converges weakly to $S$ as $n\to\infty$.
By Lemmas \ref{lem:measLSC} and \ref{lem:disDeltaLSC}, we have
\begin{align*}
\max\{\dis_\Delta(S),1-\pi(S)\}&\le \max\{\liminf_{n\to\infty}\dis_\Delta(S_{\iota(n)}),\liminf_{n\to\infty}(1-\pi(S_{\iota(n)}))\}\\
&\le\liminf_{n\to\infty}\max\{\dis_\Delta(S_{\iota(n)}),1-\pi(S_{\iota(n)})\}\\
&\le\liminf_{n\to\infty}\left(\dis_\Delta\pi+\frac1{\iota(n)}\right)\\
&=\dis_\Delta\pi.
\end{align*}
This completes the proof.
\end{proof}

\begin{proof}[Proof of Theorem \ref{thm:prohOpt}]
By Lemmas \ref{lem:disDeltaPiOpt} and \ref{lem:disDeltaS_Opt}, there exist a transport plan $\pi\in\Pi(\mu,\nu)$ and a closed set $S\subset X\times X$ such that
\[
\dP(\mu,\nu)=\max\{\dis_\Delta(S),1-\pi(S)\}.
\]
This implies
\[
\dP(\mu,\nu)=\min_{\pi\in\Pi(\mu,\nu), S\subset X\times X}\max\{\dis_\Delta(S),1-\pi(S)\}
\]
by Theorem \ref{thm:Strassen}.
\end{proof}

\section{Box distance via optimal transport}\label{sec:boxOpt}
In this section, we give the proof of Theorem \ref{thm:boxOpt}.
The proof is very similar to Theorem \ref{thm:prohOpt} so we often omit similar parts.
The following symbol is similar to in the Definition \ref{dfn:disDelta}.
\begin{dfn}[Distortion of a transport plan]\label{dfn:disTrans}
Let $m_X$ and $m_Y$ be two Borel probability measures on $X$ and $Y$, respectively.
We define {\it the distortion} of a transport plan $\pi\in\Pi(m_X,m_Y)$ between $m_X$ and $m_Y$ by
\[
\dis\pi:=\inf_S \max\{\dis S,1-\pi(S)\},
\]
where $S\subset X\times Y$ runs over all closed subsets.
\end{dfn}
The parameters are rewritten into the transport plans by the following Lemma \ref{lem:paraPlan}.
\begin{lem}\label{lem:paraPlan}
Let $X$ and $Y$ be two mm-spaces. Then we have
\[
\Pi(m_X,m_Y)=\{(\varphi,\psi)_*\cL^1\mid \varphi:I\to X,\ \psi:I\to Y \},
\]
where $\varphi$ and $\psi$ are two parameters of $X$ and $Y$, respectively.
\end{lem}
\begin{proof}
Take any two parameters $\varphi:I\to X$ and $\psi:I\to Y$ of two mm-spaces $X$ and $Y$, respectively.
Since we have
\[
(\pr_1)_*(\varphi,\psi)_*\cL^1=(\pr_1\circ(\varphi,\psi))_*\cL^1=\varphi_*\cL^1=m_X
\]
and
\[
(\pr_2)_*(\varphi,\psi)_*\cL^1=(\pr_2\circ(\varphi,\psi))_*\cL^1=\psi_*\cL^1=m_Y,
\]
we obtain $(\varphi,\psi)_*\cL^1\in\Pi(m_X,m_Y)$.

Conversely, we take any transport plan $\pi\in\Pi(m_X,m_Y)$.
There exists a parameter $\Phi:I\to X\times Y$ of the measure space $(X\times Y,\pi)$ by Lemma 4.2 in \cite{Shioya:mmg}.
We put $\varphi:=\pr_1\circ\Phi$ and $\psi:=\pr_2\circ\Phi$. The maps $\varphi:=\pr_1\circ\Phi$ and $\psi:=\pr_2\circ\Phi$ are two parameters of $X$ and $Y$, respectively.
Now we have
\[
\pi=\Phi_*\cL^1=(\varphi,\psi)_*\cL^1.
\]
This completes the proof.
\end{proof}

\begin{lem}\label{lem:disBox}
Let $\varphi:I\to X$ and $\psi:I\to Y$ be two parameters of two mm-spaces $X$ and $Y$, respectively. Then we have
\[
\square(\varphi^*d_X,\psi^*d_Y)=\dis ((\varphi,\psi)_*\cL^1).
\]
\end{lem}
\begin{proof}
Take any real number $\ep >\square(\varphi^*d_X,\psi^*d_Y)$. There exists a Borel set $I_0\subset I$ such that $\cL^1(I_0)\ge 1-\ep$ and
$
|\varphi^*d_X(s,t)-\psi^*d_Y(s,t)|\le \ep
$
for any $s,t\in I_0$. We define
\[
S:=\overline{(\varphi,\psi)(I_0)}.
\]
Then we have $\dis S\le \ep$ and $(\varphi,\psi)_*\cL^1(S)\ge\cL^1(I_0)\ge1-\ep$. Hence we obtain $\dis((\varphi,\psi)_*\cL^1)\le\ep$. This implies that $\square(\varphi^*d_X,\psi^*d_Y)\ge\dis ((\varphi,\psi)_*\cL^1)$. 

Conversely, take any real number $\ep>\dis ((\varphi,\psi)_*\cL^1)$. There exists a Borel set $S\subset X\times Y$ such that $(\varphi,\psi)_*\cL^1(S)\ge 1-\ep$ and $\dis S\le \ep$. We define $I_0:=(\varphi,\psi)^{-1}(S)$. We have $\cL^1(I_0)=(\varphi,\psi)_*\cL^1(S)\ge 1-\ep$. Since $\dis S\le \ep$, we have $|\varphi^*d_X(s,t)-\psi^*d_Y(s,t)|\le \ep$ for any $s,t\in I_0$. Hence we obtain $\square(\varphi^*d_X,\psi^*d_Y)\le\ep$. This implies that $\square(\varphi^*d_X,\psi^*d_Y)\le\dis ((\varphi,\psi)_*\cL^1)$. This completes the proof.
\end{proof}
\begin{prop}\label{prop:transportBox}
Let $X$ and $Y$ be two mm-spaces. Then we have
\[
\square(X,Y)=\inf_{\pi\in\Pi(m_X,m_Y)}\dis\pi.
\]
\end{prop}
\begin{proof}
By Lemmas \ref{lem:paraPlan} and \ref{lem:disBox}, we have
\begin{align*}
\square(X,Y)&=\inf_{\varphi,\psi}\square(\varphi^*d_X,\psi^*d_Y)\\
&=\inf_{\varphi,\psi}\dis ((\varphi,\psi)_*\cL^1)\\
&=\inf_{\pi\in\Pi(m_X,m_Y)}\dis\pi.
\end{align*}
This completes the proof.
\end{proof}
\begin{lem}\label{lem:disIneq}
Let $X$ and $Y$ be two metric spaces. Let $X\times Y$ be the product space equipped with the $l^1$-metric. Then we have
\[
\dis U_t(S)\le \dis S+2t.
\]
\end{lem}
\begin{proof}
Take any two points $(x_1,y_1),(x_2,y_2)\in U_t(S)$. It is sufficient to prove that
\begin{equation}\label{lem:disIneq:goal}
|d_X(x_1,x_2)-d_Y(y_1,y_2)|\le \dis S+2t.
\end{equation}
From symmetry, we may assume that $d_X(x_1,x_2)\ge d_Y(y_1,y_2)$. By $(x_i,y_i)\in U_t(S)$, there exists $(x'_i,y'_i)\in S$ such that
\[
d_X(x_i,x'_i)+d_Y(y_i,y'_i)<t
\]
 for $i=1,2$. Now we have
\begin{align*}
&|d_X(x_1,x_2)-d_Y(y_1,y_2)|=d_X(x_1,x_2)-d_Y(y_1,y_2)\\
&\le d_X(x'_1,x'_2)-d_Y(y'_1,y'_2)+\sum_{i=1}^2\{d_X(x_i,x'_i)+d_Y(y_i,y'_i)\}\\
&<|d_X(x'_1,x'_2)-d_Y(y'_1,y'_2)|+2t
\le \dis S+2t.
\end{align*}
This completes the proof.
\end{proof}

\begin{lem}\label{lem:disConti}
Let $X$ and $Y$ be two metric spaces. Then we have
\[
|\dis\pi-\dis\pi'|\le 2\dP(\pi,\pi')
\]
for any two Borel probability measures $\pi$ and $\pi'$ on $X\times Y$, where a metric on $X\times Y$ is the $l^1$-metric.
\end{lem}

\begin{proof}
By Lemma \ref{lem:disIneq}, we prove in the same way as Lemma \ref{lem:disDiagonalConti}.
\end{proof}

\begin{lem}\label{lem:optBox1}
Let $X$ and $Y$ be two mm-spaces. Then there exists a transport plan $\pi\in\Pi(m_X,m_Y)$ such that
$\square(X,Y)=\dis \pi$.
\end{lem}

\begin{proof}
This follows from Proposition \ref{prop:transportBox} and Lemma \ref{lem:disConti}, and the weak compactness of $\Pi(m_X,m_Y)$.
\end{proof}

\begin{lem}\label{lem:disLSC}
Let $X$ and $Y$ be two metric spaces. Let $X\times Y$ be the product space equipped with the $l^1$-metric. Let $S,S_n\subset X\times Y$, $n=1,2,\dots$, be closed subsets. If $S_n$ converges weakly to $S$ as $n\to\infty$, then we have
\[
\dis S\le \liminf_{n\to\infty}\dis S_n.
\]
\end{lem}
\begin{proof}
By Lemma \ref{lem:disIneq}, we prove in the same way as Lemma \ref{lem:disDeltaLSC}.
\end{proof}

\begin{lem}\label{lem:boxOptimalS}
Let $X$ and $Y$ be two mm-spaces and $\pi$ a transport plan between $m_X$ and $m_Y$. Then there exists a closed subset $S\subset X\times Y$ such that $\dis \pi=\max\{1-\pi(S),\dis S \}$.
\end{lem}
\begin{proof}
By Lemma \ref{lem:disLSC}, we prove in the same way as Lemma \ref{lem:disDeltaS_Opt}.
\end{proof}

\begin{proof}[Proof of Theorem  {\rm \ref{thm:boxOpt}}]
Theorem \ref{thm:boxOpt} follows from Proposition \ref{prop:transportBox} and Lemmas \ref{lem:optBox1} and \ref{lem:boxOptimalS}.
\end{proof}
\section{Observable distance via optimal transport}
\begin{dfn}
Let $X$ and $Y$ be two mm-spaces.
We put
\[
\dconcpi\pi(X,Y):=d_\mathrm{H}^{\dkf\pi}(\pr_1^*\Lip_1(X),\pr_2^*\Lip_1(Y))
\]
for a transport plan $\pi\in\Pi(m_X,m_Y)$.
\end{dfn}

The goal of this section is to prove the following Theorem \ref{thm:optConc}.
\begin{thm}\label{thm:optConc}
We have
\[
\dconc(X,Y)=\min_{\pi\in\Pi(m_X,m_Y)}\dconcpi\pi(X,Y)
\]
for any two mm-spaces $X$ and $Y$.
\end{thm}
\begin{lem}\label{lem:kfCoupParameter}
Let $\varphi:I\to X$ and $\psi:I\to Y$ be two parameters of two mm-spaces $X$ and $Y$, respectively. Then we have
\[
\dkf{(\varphi,\psi)_*\leb}(\mathrm{pr}_1^*f,\mathrm{pr}_2^*g)=\dkf{\leb}(\varphi^*f,\psi^*g),
\]
where we put $F^*G:=G\circ F$ for any maps $F\colon A\to B$ and $G\colon B\to C$.
\end{lem}

\begin{proof}
We have
\begin{align*}
\dkf{(\varphi,\psi)_*\leb}(\mathrm{pr}_1^*f,\mathrm{pr}_2^*g)&=\dkf{\leb}((\varphi,\psi)^*(\mathrm{pr}_1^*f),(\varphi,\psi)^*(\mathrm{pr}_2^*g))\\
&=\dkf{\leb}((\varphi,\psi)^*\mathrm{pr}_1)^*f,((\varphi,\psi)^*\mathrm{pr}_2)^*g)\\
&=\dkf{\leb}(\varphi^*f,\psi^*g).
\end{align*}
This completes the proof.
\end{proof}

\begin{lem}\label{lem:dconcCoupParameter}
Let $\varphi:I\to X$ and $\psi:I\to Y$ be two parameters of two mm-spaces $X$ and $Y$ respectively. Then we have
\[
d_\mathrm{H}^{\rm KF}(\varphi^*\Lip_1(X),\psi^*\Lip_1(Y))=\dconcpi{(\varphi,\psi)_*\leb}(X,Y).
\]
\end{lem}

\begin{proof}
This follows from Lemma \ref{lem:kfCoupParameter}.
\end{proof}

\begin{prop}\label{prop:transportConc}
We have
\[
\dconc(X,Y)=\inf_{\pi\in\Pi(m_X,m_Y)}\dconcpi\pi(X,Y)
\]
for any two mm-spaces $X$ and $Y$.
\end{prop}
\begin{proof}
By Lemmas \ref{lem:paraPlan} and \ref{lem:dconcCoupParameter}, we have
\begin{align*}
\dconc(X,Y)&=\inf_{\varphi,\psi}d_\mathrm{H}^{\rm KF}(\varphi^*\Lip_1(X),\psi^*\Lip_1(Y))\\
&=\inf_{\varphi,\psi}\dconcpi{(\varphi,\psi)_*\leb}(X,Y)\\
&=\inf_{\pi\in\Pi(m_X,m_Y)}\dconcpi\pi(X,Y).
\end{align*}
This completes the proof.
\end{proof}

\begin{lem}\label{lem:kfWeakConti}
Let $X$ and $Y$ be two metric spaces. Let $f$ and $g$ be two 1-Lipschitz maps from $X$ to $Y$.
Then we have
\[
|\dkf\mu(f,g)-\dkf\nu(f,g)|\le 2\dP(\mu,\nu)
\]
for any two Borel probability measures $\mu$ and $\nu$ on $X$.
\end{lem}

\begin{proof}
Take any two real numbers $r>\dP(\mu,\nu)$ and $s>\dkf\nu(f,g)$.
We put $F(x):=d_Y(f(x),g(x))$ for any point $x\in X$.
$F$ is a Lipschitz function with Lipschitz constant at most $2$.
Hence we have $U_r(F^{-1}(A))\subset F^{-1}(U_{2r}(A))$ for any subset $A\subset \R$.
Since $s>\dkf\nu(f,g)$, we have 
\[
\nu(F^{-1}((s,\infty)))\le s.
\]
Therefore we have
\begin{align*}
\mu(F^{-1}((2r+s,\infty)))&\le \nu(U_r(F^{-1}((2r+s,\infty))))+r\\
&\le \nu(F^{-1}(U_{2r}((2r+s,\infty))))+r\\
&=\nu(F^{-1}((s,\infty)))+r\\
&\le s+r.
\end{align*}
Hence we obtain $\dkf\mu(f,g)\le \dkf\nu(f,g)+2\dP(\mu,\nu)$. By exchanging $\mu$ for $\nu$, we also obtain $\dkf\nu(f,g)\le \dkf\mu(f,g)+2\dP(\mu,\nu)$.
\end{proof}

\begin{lem}\label{lem:dconcWeakConti}
Let $X$ and $Y$ be two metric spaces. We assume that $X\times Y$ is a metric space satisfying that the first projection $\pr_1:X\times Y\to X$ and the second projection $\pr_2:X\times Y\to Y$ are both 1-Lipschitz. Then we have
\[
|\dconcpi\pi(X,Y)-\dconcpi{\pi'}(X,Y)|\le 2\dP(\pi,\pi')
\]
for any two Borel probability measures $\pi$ and $\pi'$ on $X\times Y$.
\end{lem}

\begin{proof}
By symmetry, it suffices to prove $\dconcpi{\pi'}(X,Y)\le\dconcpi{\pi}(X,Y)+2\dP(\pi,\pi')$.
Put $r:=\dP(\pi,\pi')$ and take any real number $\varepsilon>\dconcpi\pi(X,Y)$.
Let us prove
\begin{equation}\label{eq:dconcWeakConti}
\pr_1^*\Lip_1(X)\subset U_{\varepsilon+2r}^{\dkf{\pi'}}(\pr_2^*\Lip_1(Y)).
\end{equation}

We take any function $f\in\Lip_1(X)$.
Since $\varepsilon>\dconcpi\pi(X,Y)$, there exists $g\in\Lip_1(Y)$ such that $\dkf\pi(\pr_1^*f,\pr_2^*g)<\varepsilon$.
By Lemma \ref{lem:kfWeakConti}, we have
\begin{align*}
\dkf{\pi'}(\pr_1^*f,\pr_2^*g)&\le \dkf\pi(\pr_1^*f,\pr_2^*g)+2r\\
&<\varepsilon+2r.
\end{align*}
Hence we obtain \eqref{eq:dconcWeakConti}.
We also obtain $\pr_2^*\Lip_1(Y)\subset U_{\varepsilon+2r}^{\dkf{\pi'}}(\pr_1^*\Lip_1(X))$ in the same way.
Since we have $\dconcpi{\pi'}(X,Y)\le \varepsilon +2r$, we obtain $\dconcpi{\pi'}(X,Y)\le\dconcpi{\pi}(X,Y)+2\dP(\pi,\pi')$.
This completes the proof.
\end{proof}

\begin{lem}\label{lem:ExistsConcOpt}
Let $X$ and $Y$ be two mm-spaces. Then there exists a transport plan $\pi\in\Pi(m_X,m_Y)$ such that
$\dconc(X,Y)=\dconcpi\pi(X,Y)$.
\end{lem}

\begin{proof}
This follows from Proposition \ref{prop:transportConc}, Lemma \ref{lem:dconcWeakConti}, and the weak compactness of $\Pi(m_X,m_Y)$.
\end{proof}
\begin{proof}[Proof of Theorem  {\rm \ref{thm:optConc}}]
This follows from Theorem \ref{thm:optConc} and Lemma \ref{lem:ExistsConcOpt}.
\end{proof}

\section{Eurandom distance via optimal transport}\label{sec:Eur_opt}
The definition of Eurandom metric $d_{\rm Eur}$ is very similar to of box metric $\square$. In this section, we show that the Eurandom metric has optimal transports. 

\begin{dfn}[Eurandom metric]
The Eurandom distance $d_{\rm Eur}$ between two mm-spaces $X$ and $Y$ is defined by
\[
d_{\rm Eur}(X,Y):=\inf_{\pi\in\Pi(m_X,m_Y)}\inf_{\varepsilon\ge 0}\max\{\varepsilon, \pi\otimes\pi(|d_X-d_Y|>\varepsilon)\},
\]
where the set
\[
\{(x_1,y_1,x_2,y_2);|d_X(x_1,x_2)-d_Y(y_1,y_2)|>\varepsilon \}
\]
is abbreviated to
\[
|d_X-d_Y|>\varepsilon.
\]
The product measure of two measures $\mu$ and $\nu$ is indicated as $\mu\otimes \nu$.
\end{dfn}

We define
\[
\dis_{\rm Eur}\mu:=\inf_{\varepsilon\ge 0}\max\{\varepsilon, \mu(|d_X-d_Y|>\varepsilon)\}
\]
for any Borel probability measure $\mu$ on $X\times Y\times X\times Y$.

\begin{lem}\label{lem:disEur}
Let $X$ and $Y$ be two metric spaces and $X\times Y\times X\times Y$ be the product space equipped with  $l^1$-metric. Let $\mu$ and $\nu$ be two Borel measures on $X\times Y\times X\times Y$. Then we have
\[
|\dis_{\rm Eur}\mu -\dis_{\rm Eur}\nu|\le \dP(\mu,\nu).
\]
\end{lem}
\begin{proof}
By symmetry, it is sufficient to prove 
\begin{equation}\label{eq1:lem:disEur}
\dis_{\rm Eur}\mu\le \dis_{\rm Eur}\nu+\dP(\mu,\nu).
\end{equation}
Take any real number $s>\dis_{\rm Eur}\nu$ and $t>\dP(\mu,\nu)$.
Since $s>\dis_{\rm Eur}\nu$, we have $\nu(|d_X-d_Y|>s)<s$. Then we have
\begin{align*}
\mu(|d_X-d_Y|>s+t)&\le \nu(U_t(|d_X-d_Y|>s+t))+t\\
&\le \nu(|d_X-d_Y|>s)+t\\
&< s+t
\end{align*}
and this implies \eqref{eq1:lem:disEur}. This completes the proof.
\end{proof}

\begin{thm}\label{thm:EurOpt}
Let $X$ and $Y$ be two mm-spaces. Then we have
\[
d_{\rm Eur}(X,Y)=\min_{\pi\in\Pi(m_X,m_Y)}\min_{\varepsilon\ge 0}\max\{\varepsilon, \pi\otimes\pi(|d_X-d_Y|>\varepsilon)\}.
\]
\end{thm}
\begin{proof}
By the definition of $d_{\rm Eur}(X,Y)$, there exists $\{\pi_n\}\subset \Pi(m_X,m_Y)$ such that
$\dis_{\rm Eur}(\pi_n\otimes \pi_n)\to d_{\rm Eur}(X,Y)$. Since the set $\Pi(X,Y)$ is compact, we may assume that $\pi_n$ converges to some $\pi\in \Pi(X,Y)$. Since $\pi_n\otimes \pi_n$ converges to $\pi\otimes\pi$ as $n\to\infty$, we have $\dis_{\rm Eur}(\pi_n\otimes \pi_n)\to\dis_{\rm Eur}(\pi\otimes \pi)$
by Lemma \ref{lem:disEur}. This implies $d_{\rm Eur}(X,Y)=\dis_{\rm Eur}(\pi\otimes \pi)$.
Since the function $\varepsilon \mapsto \pi\otimes\pi(|d_X-d_Y|>\varepsilon)$ is non-increasing and right-continuous, there exists a real number $\varepsilon>0$ such that $\dis_{\rm Eur}(\pi\otimes \pi)=\max\{\varepsilon, \pi\otimes\pi(|d_X-d_Y|>\varepsilon)\}$. This completes the proof.
\end{proof}

\section{Applications}
As an application of Theorem \ref{thm:boxOpt}, we easily prove that the function $\square$ is a distance function.
In this section, we give the proof of Theorem \ref{thm:boxIsMetric}. We also give the proof of that the function $\dconc$ is a distance function (Theorems \ref{thm:obsTriangle} and \ref{thm:dconcNonDeg}). Note that it is already known that the functions $\square$ and $\dconc$ are both metrics by \cite{Gmv:green} (cf. \cite{Shioya:mmg}). The nondegeneracy of the Eurandom metric is also proved by using transport plan in this section.

\subsection{Nondegeneracy of box and observable metrics}
First, we prove the nondegeneracy of the box metric $\square$ by using optimal transport. Therefore, we prove Theorem \ref{thm:boxNondegenerate} below. 
For the proof of Theorem \ref{thm:boxNondegenerate}, we prepare the following Lemmas \ref{lem:boxNonDeg} and \ref{lem:disToEq}. The key lemma is Lemma \ref{lem:disToEq} and we only use Lemma \ref{lem:boxNonDeg} to prove Lemma \ref{lem:disToEq}.
\begin{lem}\label{lem:boxNonDeg}
Let $X$ and $Y$ be two mm-spaces and $\pi$ a transport plan between $m_X$ and $m_Y$. We assume that a map $f:\supp m_X\to \supp m_Y$ satisfies
\begin{equation}\label{lem:boxNonDeg:eq1}
\supp\pi\subset\{(x,f(x))\mid x\in\supp m_X \}.
\end{equation}
Then we have $f_*m_X=m_Y$.
\end{lem}
\begin{proof}
Take any Borel set $A\subset X$ and $B\subset Y$.
By \eqref{lem:boxNonDeg:eq1}, we have
\[
(A\times B)\cap\supp\pi=((A\cap f^{-1}(B))\times Y)\cap \supp\pi.
\]
Since we have
\begin{align*}
\pi(A\times B)&=\pi((A\cap f^{-1}(B))\times Y)\\
&=m_X(A\cap f^{-1}(B))\\
&=(\id_X,f)_*m_X(A\times B),
\end{align*}
we obtain $\pi=(\id_X,f)_*m_X$.
Then we have $m_Y=(\pr_2)_*\pi=f_*m_X$.
This completes the proof.
\end{proof}

\begin{lem}\label{lem:disToEq}
Let $X$ and $Y$ be two mm-spaces. If there exists a transport plan $\pi\in\Pi(m_X,m_Y)$ such that
$\dis\supp\pi=0$,
then $X$ and $Y$ are mm-isomorphic to each other.
\end{lem}
\begin{proof}
We define a map $f:\supp m_X\to \supp m_Y$ by
$(x,f(x))\in\supp\pi$ for any $x\in \supp m_X$. Let us prove the well-definedness of $f$. First we prove the uniqueness. Take any $x\in\supp m_X$ and any $y, y'\in \supp m_Y$ with $(x,y), (x,y')\in\supp\pi$. Since
\[
d_Y(y,y')=|d_X(x,x)-d_Y(y,y')|\le \dis\supp\pi=0,
\]
we obtain $y=y'$. Next we prove the existence.
Take any $x\in \supp m_X$.
Since
\[
\supp m_X=\supp (\pr_1)_*\pi=\overline{\pr_1(\supp\pi)},
\]
there exists a sequence $\{(x_n,y_n)\}_n\subset \supp\pi$ such that $x_n\to x$ as $n\to\infty$.
Since we have
\[
|d_Y(y_m,y_n)-d_X(x_m,x_n)|\le \dis\supp\pi=0,
\]
we have $d_Y(y_m,y_n)=d_X(x_m,x_n)\to 0$ as $m,n\to\infty$.
Therefore $\{y_n\}_n$ is a Cauchy sequence and there exists $y\in Y$ such that $y_n\to y$ as $n\to\infty$. Since $\supp\pi$ is closed, we have $(x,y)\in \supp\pi$ and we have
\[
y\in \pr_2(\supp\pi)\subset \supp(\pr_2)_*\pi=\supp m_Y.
\]
Therefore we obtain the well-definedness of $f$. Let us prove that $f$ is an mm-isomorphism. For any $x, x'\in\supp m_X$, we have
\[
|d_Y(f(x),f(x'))-d_X(x,x')|\le \dis\supp\pi=0.
\]
By the definition of $f$, we have
$\supp\pi=\{(x,f(x))\mid x\in\supp m_X \}$.
This implies
$m_Y=f_*m_X$ by Lemma \ref{lem:boxNonDeg}.
\end{proof}

\begin{thm}\label{thm:boxNondegenerate}
Let $X$ and $Y$ be two mm-spaces. If $\square(X,Y)=0$, then $X$ and $Y$ are mm-isomorphic to each other.
\end{thm}
\begin{proof}
We assume that $\square(X,Y)=0$.
By theorem \ref{thm:boxOpt}, there exist a transport plan $\pi\in\Pi(X,Y)$ and a closed subset $S\subset X\times Y$ such that $1-\pi(S)=0$ and $\dis S=0$. Since $1-\pi(S)=0$, we have $\supp \pi\subset S$. This implies
\[
\dis \supp\pi\le \dis S=0.
\]
This implies that $X$ and $Y$ are mm-isomorphic by Lemma \ref{lem:disToEq}.
\end{proof}

Next, we prove the nondegeneracy of the function $\dconc$ by using optimal transport and Lemma \ref{lem:disToEq}.
Let us prepare Lemma \ref{lem0:obsNonDeg} for the proof of Theorem \ref{thm:dconcNonDeg} below.

\begin{lem}\label{lem0:obsNonDeg}
Let $X$ and $Y$ be two mm-spaces and let $\pr_1\colon X\times Y\to X$ be the projection map. Then the set 
$(\pr_1)^*\Lip_1(X)$ is closed with respect to $\dkf\pi$ for  any $\pi\in\Pi(m_X,m_Y)$.
\end{lem}

\begin{proof}
Take any sequence $(\varphi_n)_{n\in\N}\in (\Lip_1(X))^\N$ and any measurable function $F:X\times Y\to\R$ satisfying $\dkf\pi((\pr_1)^*\varphi_n,F)\to 0$ as $n\to\infty$.
Then there exists an increasing function $\iota:\N\to\N$ such that
\[
(\pr_1)^*\varphi_{\iota(n)}\to F \text{ \quad $\pi$- a.e. on $X$.}
\]
This implies that there exists a Borel set $S\subset X\times Y$ with $\pi(S)=1$ such that
$(\pr_1)^*\varphi_{\iota(n)}\to F $ on $S$.
Now we have
\begin{equation}\label{lem0:obsNonDeg:eq1}
F(x,y)=F(x,y')
\end{equation}
for any $x\in X$ and $y,y'\in Y$ with $(x,y),(x,y')\in S$.
We define a function $f:\pr_1(X)\to\R$ as
\[
f(x):=F(x,y) \text{\quad for $(x,y)\in S$}.
\]
By \eqref{lem0:obsNonDeg:eq1}, $f$ is well-defined. The function $f$ is $1$-Lipschitz because $\varphi_n$ is $1$-Lipschitz for any $n\in\N$. The function $f$ is extended continuously such that the domain of $f$ is $\supp m_X$ because $\supp m_X\subset \overline{\pr_1(S)}$. Since we have $F=(\pr_1)^*f$ on $S$, we obtain $F\in (\pr_1)^*\Lip_1(X)$. This completes the proof.
\end{proof}

\begin{thm}\label{thm:dconcNonDeg}
Let $X$ and $Y$ be two mm-spaces. If $\dconc(X,Y)=0$, then $X$ and $Y$ are mm-isomorphic to each other.
\end{thm}
\begin{proof}
We assume that $\dconc(X,Y)=0$. By Theorem \ref{thm:optConc}, there exists $\pi\in\Pi(m_X,m_Y)$ such that $\dconcpi\pi(X,Y)=0$.
By Lemma \ref{lem0:obsNonDeg}, we have
\begin{equation}\label{lem1:obsNonDeg:eq1}
(\pr_1)^*\Lip_1(X)=(\pr_2)^*\Lip_1(Y).
\end{equation}
Take any two points $(x,y), (x',y')\in\supp\pi$.
Let us prove
\begin{equation}\label{lem1:obsNonDeg:eq2}
d_Y(y,y')\le d_X(x,x').
\end{equation}
By $d_Y(y,\cdot)\in \Lip_1(Y)$ and \eqref{lem1:obsNonDeg:eq1}, there exists $\varphi\in\Lip_1(X)$ such that
\[
(\pr_1)^*\varphi=(\pr_2)^*d_Y(y,\cdot) \text{ \quad $\pi$-a.e. on $X$}
\]
Since $(\pr_1)^*\varphi$ and $(\pr_2)^*d_Y(y,\cdot)$ are continuous, we have
\begin{equation}\label{lem1:obsNonDeg:eq3}
(\pr_1)^*\varphi=(\pr_2)^*d_Y(y,\cdot) \text{ \quad on $\supp\pi$.}
\end{equation}
By $(x,y),(x',y')\in\supp\pi$ and \eqref{lem1:obsNonDeg:eq3}, we have
\begin{align*}
d_Y(y,y')&=d_Y(y,y')-d_Y(y,y)\\
&=((\pr_2)^*d_Y(y,\cdot))(x',y')-((\pr_2)^*d_Y(y,\cdot))(x,y)\\
&=((\pr_1)^*\varphi)(x',y')-((\pr_2)^*\varphi)(x,y)\\
&=\varphi(x')-\varphi(x).
\end{align*}
This implies \eqref{lem1:obsNonDeg:eq2} because $\varphi$ is $1$-Lipschitz.
Similarly we have
\[
d_X(x,x')\le d_Y(y,y').
\]
Then we obtain $\dis\supp\pi=0$. This completes the proof by Lemma \ref{lem:disToEq}.
\end{proof}

The nondegeneracy of the Eurandom metric $d_{\rm Eur}$ is also proved by Lemma \ref{lem:disToEq}.
\begin{thm}
Let $X$ and $Y$ be two mm-spaces. If $d_{\rm Eur}(X,Y)=0$, then $X$ and $Y$ are mm-isomorphic to each other.
\end{thm}
\begin{proof}
We assume that $d_{\rm Eur}(X,Y)=0$. By Theorem \ref{thm:EurOpt}, there exists $\pi\in\Pi(m_X,m_Y)$ such that
$\pi\otimes \pi(|d_X-d_Y|=0)=1$.
This implies
\[
\supp\pi\times \supp\pi=\supp(\pi\otimes \pi)\subset \{|d_X-d_Y|=0\}.
\]
This completes the proof by Lemma \ref{lem:disToEq}.
\end{proof}

\subsection{Triangle inequalities of box and observable metrics}
The triangle inequality of the box metric $\square$ is proved by using the gluing lemma of measures.
Remark that we use Proposition \ref{prop:transportBox} in the proof of Theorem \ref{thm:boxTriangle} but we do not need to use Theorem \ref{thm:boxOpt}. The triangle inequality of the function $\dconc$ (Theorem \ref{thm:obsTriangle}) is also proved by using the gluing lemma. The triangle inequality of the Eurandom metric $d_{\rm Eur}$ is also proved by using the gluing lemma easily, so it is omitted.
\begin{thm}[\cite{Gmv:green}, cf.~Theorem 4.10 in \cite{Shioya:mmg}]\label{thm:boxTriangle}
The function $\square$ satisfies the triangle inequality.
\end{thm}
\begin{proof}
Let $X$, $Y$ and $Z$ be three mm-spaces. Let us prove
\begin{equation}\label{thm:boxTriangle:eq1}
\square(X,Z)\le \square(X,Y)+\square(Y,Z).
\end{equation}
Take any real number $\alpha>\square(X,Y)$ and $\beta>\square(Y,Z)$. 
By Proposition \ref{prop:transportBox}, there exist $\pi_{XY}\in \Pi(X,Y)$ and $S\subset X\times Y$ such that
\[
\max\{\dis S,1-\pi_{XY}(S)\}<\alpha.
\]
Similarly, there exist $\pi_{YZ}\in \Pi(Y,Z)$ and $T\subset Y\times Z$ such that
\[
\max\{\dis T,1-\pi_{YZ}(T)\}<\beta.
\]
Now we put
\[
T\circ S:=\pr_{13}((S\times Z)\cap (X\times T)),
\]
where $\pr_{ij}:=(\pr_i,\pr_j)$ is the pair of projections for any $i,j=1,2,3$.
Then we have
\begin{equation}\label{thm:boxTriangle:eq2}
\dis (\overline{T\circ S})=\dis (T\circ S)\le \dis T+\dis S<\alpha+\beta,
\end{equation}
where $\overline{T\circ S}$ is the closure of $T\circ S$.
By the gluing lemma, there exists a probability measure $\pi_{XYZ}$ on $X\times Y\times Z$ such that $\pi_{XY}=(\pr_{12})_*\pi_{XYZ}$ and $\pi_{YZ}=(\pr_{23})_*\pi_{XYZ}$. We put $\pi_{XZ}:=(\pr_{13})_*\pi_{XYZ}$ and we have
\begin{equation}\label{thm:boxTriangle:eq3}
\begin{aligned}
\pi_{XZ}(\overline{T\circ S})&=\pi_{XYZ}(\pr_{13}^{-1}(\overline{T\circ S}))\\
&\ge \pi_{XYZ}(\pr_{12}^{-1}(S)\cap \pr_{23}^{-1}(T))\\
&\ge \pi_{XY}(S)+\pi_{YZ}(T)-1\\
&>(1-\alpha)+(1-\beta)-1=1-\alpha-\beta.
\end{aligned}
\end{equation}
By \eqref{thm:boxTriangle:eq2} and \eqref{thm:boxTriangle:eq3}, we have
$\square(X,Z)\le \alpha+\beta$. This implies \eqref{thm:boxTriangle:eq1}.
\end{proof}

\begin{thm}[\cite{Gmv:green}, cf. Theorem 5.13 in \cite{Shioya:mmg}]\label{thm:obsTriangle}
The observable distance function $\dconc$ satisfies the triangle inequality.
\end{thm}
\begin{proof}
Let $X$, $Y$ and $Z$ be three mm-spaces. Let us prove
\begin{equation}\label{thm:obsTriangle:eq1}
\dconc(X,Z)\le \dconc(X,Y)+\dconc(Y,Z).
\end{equation}
Take any real number $\alpha>\dconc(X,Y)$ and $\beta>\dconc(Y,Z)$. 

By Proposition \ref{prop:transportConc}, there exists $\pi_{XY}\in \Pi(X,Y)$ such that
\begin{equation}\label{thm:obsTriangle:start1}
\dconcpi{\pi_{XY}}(X,Y)<\alpha.
\end{equation}
Similarly, there exists $\pi_{YZ}\in \Pi(Y,Z)$ such that
\begin{equation}\label{thm:obsTriangle:start2}
\dconcpi{\pi_{YZ}}(Y,Z)<\beta.
\end{equation}
By the gluing lemma, there exists a probability measure $\pi_{XYZ}$ on $X\times Y\times Z$ such that $\pi_{XY}=(\pr_{12})_*\pi_{XYZ}$ and $\pi_{YZ}=(\pr_{23})_*\pi_{XYZ}$. We put $\pi_{XZ}:=(\pr_{13})_*\pi_{XYZ}$. Let us prove
\begin{equation}\label{thm:obsTriangle:eq2}
(\pr_1)^*\Lip_1(X)\subset U_{\alpha+\beta}^{\pi_{XZ}}((\pr_2)^*\Lip_1(Z)),
\end{equation}
where the set $U_\alpha^\pi(A):=\{f:X\times Y\to\R\mid \dkf\pi(f,A)<\alpha \}$ is the open $\alpha$-neighborhood of $A$ with respect to $\dkf\pi$.
Take any $\varphi\in\Lip_1(X)$. By \eqref{thm:obsTriangle:start1}, there exists $\psi\in\Lip_1(Y)$ such that
\begin{equation}\label{thm:obsTriangle:eq3}
\dkf{\pi_{XY}}((\pr_1)^*\varphi,(\pr_2)^*\psi)<\alpha.
\end{equation}
By \eqref{thm:obsTriangle:start2}, there exists $\chi\in\Lip_1(Y)$ such that
\begin{equation}\label{thm:obsTriangle:eq4}
\dkf{\pi_{YZ}}((\pr_1)^*\psi,(\pr_2)^*\chi)<\beta.
\end{equation}
Now we have
\begin{equation}\label{thm:obsTriangle:eq5}
\begin{aligned}
&\dkf{\pi_{XZ}}((\pr_1)^*\varphi,(\pr_2)^*\chi)\\
&= \dkf{\pi_{XYZ}}((\pr_{13})^*(\pr_1)^*\varphi,(\pr_{13})^*(\pr_{2})^*\chi)\\
&=\dkf{\pi_{XYZ}}((\pr_1)^*\varphi,(\pr_3)^*\chi)\\
&\le\dkf{\pi_{XYZ}}((\pr_1)^*\varphi,(\pr_2)^*\psi)+\dkf{\pi_{XYZ}}((\pr_2)^*\psi,(\pr_3)^*\chi)\\
&=\dkf{\pi_{XYZ}}((\pr_{12})^*(\pr_1)^*\varphi,(\pr_{12})^*(\pr_{2})^*\psi)\\
&\quad+\dkf{\pi_{XYZ}}((\pr_{23})^*(\pr_1)^*\psi,(\pr_{23})^*(\pr_{2})^*\chi)\\
&=\dkf{\pi_{XY}}((\pr_1)^*\varphi,(\pr_2)^*\psi)+\dkf{\pi_{YZ}}((\pr_1)^*\psi,(\pr_2)^*\chi)\\
&<\alpha+\beta.
\end{aligned}
\end{equation}
By \eqref{thm:obsTriangle:eq5}, we obtain \eqref{thm:obsTriangle:eq2}.
Similarly, we have 
\begin{equation}\label{thm:obsTriangle:eq6}
(\pr_2)^*\Lip_1(Z)\subset U_{\alpha+\beta}^{\pi_{XZ}}((\pr_1)^*\Lip_1(X)).
\end{equation}
By \eqref{thm:obsTriangle:eq5} and \eqref{thm:obsTriangle:eq6}, we have
\[
\dconc(X,Z)\le\dconcpi{\pi_{XZ}}(X,Z)\le \alpha+\beta.
\]
This completes the proof.
\end{proof}
  

\begin{thebibliography}{9}
  \bib{Beer:Top}{book}{
   author={Beer, Gerald},
   title={Topologies on closed and closed convex sets},
   series={Mathematics and its Applications},
   volume={268},
   publisher={Kluwer Academic Publishers Group, Dordrecht},
   date={1993},
   pages={xii+340},
}

\bib{GPW}{article}{
   author={Greven, Andreas},
   author={Pfaffelhuber, Peter},
   author={Winter, Anita},
   title={Convergence in distribution of random metric measure spaces
   ($\Lambda$-coalescent measure trees)},
   journal={Probab. Theory Related Fields},
   volume={145},
   date={2009},
   number={1-2},
   pages={285--322},
}

\bib{Gmv:green}{book}{
   author={Gromov, Misha},
   title={Metric structures for Riemannian and non-Riemannian spaces},
   series={Modern Birkh\"auser Classics},
   edition={Reprint of the 2001 English edition},
   note={Based on the 1981 French original;
   With appendices by M. Katz, P. Pansu and S. Semmes;
   Translated from the French by Sean Michael Bates},
   publisher={Birkh\"auser Boston, Inc., Boston, MA},
   date={2007},
   pages={xx+585},
   isbn={978-0-8176-4582-3},
   isbn={0-8176-4582-9},
}
\bib{Loh:dGP}{article}{
   author={L\"{o}hr, Wolfgang},
   title={Equivalence of Gromov-Prohorov- and Gromov's
   $\underline\square_\lambda$-metric on the space of metric measure spaces},
   journal={Electron. Commun. Probab.},
   volume={18},
   date={2013},
   pages={no. 17, 10},
}

\bib{NkjShioya:mm-gp}{article}{
  author={Nakajima, Hiroki},
  author={Shioya, Takashi},
  title={Convergence of group actions in metric measure geometry},
  note={In preprint, arXiv: 2104.00187},
}


\bib{Shioya:mmg}{book}{
   author={Shioya, Takashi},
   title={Metric measure geometry},
   series={IRMA Lectures in Mathematics and Theoretical Physics},
   volume={25},
   note={Gromov's theory of convergence and concentration of metrics and
   measures},
   publisher={EMS Publishing House, Z\"urich},
   date={2016},
   pages={xi+182},
   isbn={978-3-03719-158-3},
}

\bib{Sturm:sp}{article}{
   author={Sturm, Karl-Theodor},
   title={The space of spaces: curvature bounds and gradient flows on the space of
metric measure spaces},
  note={arXiv:1208.0434v2}
}

\end{thebibliography}
\end{document}